\documentclass[twoside,a4paper]{amsart}
\usepackage{style}
\usepackage{mathabx}
\usepackage{tikz}
\usepackage{stmaryrd}
\usepackage{graphicx}

\usepackage[left=3cm, right=3cm]{geometry}

\usepackage{xcolor}
\usepackage{pdfsync}
\usepackage{xcolor}
\definecolor{webgreen}{rgb}{0,.5,0}
\definecolor{webbrown}{rgb}{.6,0,0}
\definecolor{RoyalBlue}{cmyk}{1, 0.50, 0, 0}
\usepackage[colorlinks=true, breaklinks=true, urlcolor=webbrown, linkcolor=RoyalBlue, citecolor=webgreen,backref=page]{hyperref}

\usepackage{courier} 
\normalfont
\usepackage{epsfig, graphicx, subfigure}
\usepackage{amsmath, amssymb}

\let\Im=\undefined
\DeclareMathOperator{\Im}{Im}

\def\ge{\geqslant}
\def\le{\leqslant}

\newtheorem{theorem}{Theorem}[section]

\newtheorem{lemma}[theorem]{Lemma}

\theoremstyle{remark}

\definecolor{darkbrown}{RGB}{150,1,33} 
\numberwithin{equation}{section}
\usepackage{hyperref}
\begin{document}

\title[Pointwise behavior  \ldots \ldots]{
Pointwise behavior of SU(1,1) nonlinear Fourier transform}

\begin{abstract} We show that $SU(1,1)$ NLFT can diverge pointwise for square-summable coefficients. As a consequence, we prove that the classical pointwise asymptotics of polynomials orthogonal on the unit circle can fail for  measures in the Szeg\H{o} class. We also discuss some special cases when the pointwise convergence holds.
\end{abstract} \vspace{1cm}

\author[Sergey A. Denisov]{Sergey A. Denisov}
\address{Department of Mathematics, University of Wisconsin-Madison, 480 Lincoln Dr., Madison, WI 53706, USA}
\email{\href{mailto:denissov@wisc.edu}{denissov@wisc.edu}}

\thanks{
This research was supported by the NSF grant DMS-2450716,  the Simons Fellowship in Mathematics, the Simons Travel Support for Mathematicians Award, and the Van Vleck Professorship Research Award. The author gratefully acknowledges the hospitality of IHES where part of this work was done.
}



\maketitle

\setcounter{tocdepth}{3}

\tableofcontents

\section{ Introduction to $SU(1,1)$ NLFT, statement of the problems, and main results.}

\medskip

\subsection{$SU(1,1)$ NLFT on $\Z$.}


 Let $\T$ and $\D$ denote the unit circle  and open unit disc in $\C$, respectively, both centered at the origin.  Let $F=\{F_n\}|_{n\in \Z}$ be a double-infinite sequence of complex numbers and $|F_n|<1, \forall n\in \Z$. Taking $N\in \Z^+$, we define a compactly supported sequence $F^{\langle N\rangle}$ by truncation: $F^{\langle N\rangle}_n=F_n\cdot \chi_{|n|\le N}, n\in \Z$, where $\chi_E$ denotes the characteristic function of a set $E$.  Define the $2\times 2$ matrices $X_n(z,F^{\langle N\rangle})$ as solution to the recursion
\begin{equation}\label{nach}
X_n=\Omega_n(F^{\langle N\rangle}) \left(\begin{array}{cc}
z& 0\\
0 &1
\end{array}\right)X_{n-1}, \quad n\in \Z\,,
\end{equation}
where \[
\Omega_n(F):=
 \frac{1}{(1-|F_{n}|^2)^{\frac 12}}
\left(\begin{array}{cc}
1& \overline{F_{n}}\\
F_{n}  &1
\end{array}\right), \quad X_{n}(z,F^{\langle N\rangle})=\left(\begin{matrix}z^{n}&0\\0&1
\end{matrix}
\right), \quad n< -N\,.
\]
Let $J:=\left(\begin{smallmatrix}1&0\\0&-1
\end{smallmatrix}
\right)$ and notice that matrix $\Omega_n(F)$ is $J$-unitary  (i.e., belongs to $U(1,1)$). The  matrix $\left(\begin{smallmatrix}
z& 0\\
0 &1
\end{smallmatrix}\right)$ is $J$-unitary for $z\in \T$, it is $J$-contraction if $|z|<1$, and $J$-expansion if $|z|>1$. Since
$
\det X_n=z^n, \, n\in \Z
$
it makes sense to renormalize 
\begin{equation}\label{red1}
X_n=:
\left(\begin{matrix}z^n&0\\0&1
\end{matrix}
\right)\widetilde X_n\,,
\end{equation}
and then
\begin{equation}\label{rec5}
\widetilde X_n= \frac{1}{(1-|F_n^{\langle N\rangle}|^2)^{\frac 12}}
\left(\begin{array}{cc}
1& \overline{F_n^{\langle N\rangle}}z^{-n}\\
F_n^{\langle N\rangle} z^n &1
\end{array}\right)\widetilde X_{n-1}, 
\end{equation}
with $\widetilde X_n=I:=\left(\begin{smallmatrix}1&0\\0&1
\end{smallmatrix}
\right)$ for $n< -N$. One can view $\widetilde X_n(z,F^{\langle N\rangle})$ as Jost solution of recursion \eqref{rec5} with normalization $\widetilde X_n(z,F^{\langle N\rangle})=I$ as $ n\to-\infty$. Now $\widetilde X_n\in SU(1,1)$ for $z\in \T$. More careful study of this new recursion  implies that (see p.5 in \cite{tt}) the matrix $\widetilde X_n$ takes the form
\[
\widetilde X_n(z,F^{\langle N\rangle})=:\left(\begin{array}{cc}
\frak{a}_n(z,F^{\langle N\rangle})& \frak{b}^{(*)}_n(z,F^{\langle N\rangle})\\
\frak{b}_n(z,F^{\langle N\rangle}) &\frak{a}^{(*)}_n(z,F^{\langle N\rangle})
\end{array}\right)\,,
\]
where \begin{equation}\label{oper1}
f^{(*)}(z):=\overline{f(\bar{z}^{-1})}.
\end{equation} 
Clearly, $f^{(*)}(z)=\overline{f(z)}$ if $z\in \T$.
Following \cite{tt} (except that we take the transposition of the matrix product used  in \cite{tt}, p.4  when defining the NLFT because we prefer to multiply the transfer matrices in the order $\ldots\cdot (\cdot)_{-N+1}\cdot (\cdot)_{-N}$ rather than  $(\cdot)_{-N}\cdot (\cdot)_{-N+1}\ldots$) we put forward a definition: \smallskip

\noindent {\bf Definition.} {\it For $F^{\langle N\rangle}$, define $SU(1,1)$ nonlinear Fourier transform (NLFT) on $\Z$ as the map 
\[
F^{\langle N\rangle}\mapsto \overbrace{F^{\langle N\rangle}}:=\left(\begin{array}{c}
\frak{a}(z,F^{\langle N\rangle})\\
\frak{b}(z,F^{\langle N\rangle}) 
\end{array}\right)\,,
\]
where $\frak{a}(z,F^{\langle N\rangle}):=\frak{a}_\infty(z,F^{\langle N\rangle})$ and 
$\frak{b}(z,F^{\langle N\rangle})=\frak{b}_\infty(z,F^{\langle N\rangle})$.}
\smallskip

 In fact, since $F^{\langle N\rangle}_n=0$ for $n>N$, we get 
\[
\left(\begin{array}{c}
\frak{a}(z,F^{\langle N\rangle})\\
\frak{b}(z,F^{\langle N\rangle}) 
\end{array}\right)=\left(\begin{array}{c}
\frak{a}_n(z,F^{\langle N\rangle})\\
\frak{b}_n(z,F^{\langle N\rangle}) 
\end{array}\right), \quad \forall n\ge N\,.
\]
The definition of NLFT can be extended to other classes of $\{F_n\}$ that decay as $|n|\to\infty$. We will provide more references later but we start by focusing on  $F\in \ell^p(\Z), p\in [1,\infty)$. It is known \cite{tt}, p.10 that $F\in \ell^1(\Z)$ implies existence of two functions $\frak{a}(z,F)$ and $\frak{b}(z,F)$, both defined on $\T$, such that $\|\frak{a}(z,F^{\langle N\rangle})-\frak{a}(z,F)\|_{A(\T)}\to 0$ and  $\|\frak{b}(z,F^{\langle N\rangle})-\frak{b}(z,F)\|_{A(\T)}\to 0$ when $N\to\infty$ and $A(\T)$ denotes the Wiener's algebra. Similarly, for $F\in \ell^p(\Z)$ with $p\in (1,2)$, the nonlinear analog of the Menshov-Paley-Zygmund theorem was established (see, \cite{tt}, p.11 and \cite{st,Den25,kov}) which yields, in particular, the following result: if $p\in (1,2)$ and $F\in \ell^p(\Z)$, then the limits 
$\lim_{N\to\infty}\frak{a}(z,F^{\langle N\rangle})$ and  $\lim_{N\to\infty}\frak{b}(z,F^{\langle N\rangle})$ exist for a.e. $z\in \T$. For $p>2$ the above limits might diverge for a.e. $z\in \T$ (see \cite{kls} for the discussion in the context of Jacobi matrices). The borderline case $p=2$ is critical, it attracted a lot of attention in harmonic analysis and approximation theory communities but the answer was not known. We formulate the following two questions:\medskip

\noindent{\bf Q1$_\Z$ (the strong version of $SU(1,1)$ pointwise convergence of NLFT on} $\Z$):  {\it Let $F\in \ell^2(\Z)$. Is it true that the limits
$\lim_{N\to\infty}\frak{a}(z,F^{\langle N\rangle})$ and  $\lim_{N\to\infty}\frak{b}(z,F^{\langle N\rangle})$ exist for a.e. $z\in \T$?}\medskip

In folklore, the assertion that those limits do actually exist  often went under the name {\it Nonlinear Carleson Conjecture (NCC)} (see, e.g., \cite{amt}, p.5 for $SU(2)$ setting). This problem is motivated by the study of the existence of wave operators for Jacobi matrices (see \cite{dy}). The $SU(1,1)$ structure of the problem ensures that 
\begin{equation}\label{sul}
\frak{a}^{(*)}(z,F^{\langle N\rangle})\frak{a}(z,F^{\langle N\rangle})-\frak{b}^{(*)}(z,F^{\langle N\rangle})\frak{b}(z,F^{\langle N\rangle})
=1,\quad z\in \C
\end{equation}
and, in particular, $|\frak{a}^{(*)}(z,F^{\langle N\rangle})|\ge 1$ for $z\in \T$.
We will use the following notation
\[
\frak{r}(z,F^{\langle N\rangle}):=\frak{b}(z,F^{\langle N\rangle})/\frak{a}^{(*)}(z,F^{\langle N\rangle})\,.
\]
Clearly, $\frak{r}(z,F^{\langle N\rangle})$  satisfies $|\frak{r}(z,F^{\langle N\rangle})|<1, z\in \T$. We suggest the following weaker version of $\bf Q1_\Z$.\medskip

\noindent{\bf Q2$_\Z$ (a weak version of $SU(1,1)$ pointwise convergence of NLFT on} $\Z$):  {\it Let $F\in \ell^2(\Z)$. Is it true that the limit
 $\lim_{N\to\infty}\frak{r}(z,F^{\langle N\rangle})$ exists for a.e. $z\in \T$?}\medskip
 
 Our first main result is the following theorem
 \begin{theorem}\label{t1}
 The answer to {\bf Q1}$_{\Z}$ is negative. If fact, there is $F\in \ell^2(\Z)$ such that the limits $\lim_{N\to\infty}\frak{a}(z,F^{\langle N\rangle})$ and  $\lim_{N\to\infty}\frak{b}(z,F^{\langle N\rangle})$ do not exist at every $z\in \T$.
 \end{theorem}\medskip

\noindent {\bf Remark.} In the theorem, we can take $F$ such that $\supp F\subset \Z^+$. Our argument does not provide an answer to {\bf Q2}$_\Z$.\medskip

 Let $\sigma$ be a probability measure on $\T$ with the infinite support (in the sense of cardinality). Denote the monic orthogonal polynomials by $\{\Phi_n(z,\sigma)\}$ and orthonormal polynomials by $\{\phi_n(z,\sigma)\}$, we use OPUC for both as shorthand. That is, 
\begin{eqnarray*}
\int_\T \Phi_n(z,\sigma)z^{-j}d\sigma=0, \quad \int_\T \phi_n(z,\sigma)z^{-j}d\sigma=0, \quad \forall j\in \{0,\ldots, n-1\},\quad\int_\T |\phi_n(z,\sigma)|^2d\sigma=1, 
\\ \deg (\Phi_n)=\deg(\phi_n)=n, \quad 
\text{coeff}_n(\Phi_n)=1, \quad \text{coeff}_n(\phi_n)>0\,,
\end{eqnarray*}
where $\text{coeff}_n(Q)$ denotes the $n$-th coefficient of the polynomial $Q$. For any polynomial $Q$, we let $Q^*:=z^n\overline{Q(\bar{z}^{-1})}$. Notice that such $\ast$-operation is different from the $(\ast)$-operation in \eqref{oper1} and it depends on $n\in \Z^+$. In fact, $Q^*=z^nQ^{(*)}$.\medskip

\noindent {\bf Definition.} {\it The measure $\sigma$ belongs to the Szeg\H{o} class ($\sigma\in \szc$) if 
\[
\int_{\T} \log w \,dm>-\infty\,,\] where $ d\sigma=wdm+d\sigma_s\,,
$
$m$ is the normalized Lebesgue probability measure on $\T$ and $\sigma_s$ is the singular part of $\sigma$.} \medskip

The existence of the limit $\lim_{n\to\infty}\phi_n^*(z,\sigma)$ for $z\in \T$ has been studied in many papers under different assumptions on measure $\sigma$ and we will discuss some results in section 3. Now, we give two relevant applications of  Theorem~\ref{t1} to the OPUC theory:\medskip

 \begin{theorem}\label{t2} There is $\sigma\in \szc$ such that the limit
 $
 \lim_{n\to\infty} \phi_n^*(z,\sigma)
 $
 does not exist for all $z\in \T$.
 \end{theorem}\medskip
 
 and\medskip
 \begin{theorem}\label{t39}There is $\sigma\in \szc$ and the sequence $\{\alpha_n\}\in \ell^2(\Z^+)$ such that the orthogonal series 
 \begin{equation}\label{lus}
 \sum_{n\in \Z^+}\alpha_n\phi_n(z,\sigma)
 \end{equation}
 diverges for all $z\in \T$.
 \end{theorem}
 \noindent {\bf Remark.} The measure $\sigma$ in the previous two theorems is actually more regular: it is absolutely continuous with continuous and positive density (see Theorem \ref{tbig}). For $\sigma=m$, we have $\phi_n=z^n$ and the series \eqref{lus} converges a.e. on $\T$ for every $\{\alpha_n\}\in \ell^2(\Z^+)$ by the Carleson theorem \cite{carles,st} which settled the Lusin's conjecture. The Theorem \ref{t39} shows that the analog of Lusin's conjecture does not hold within the Szeg\H{o} class of orthogonality measures.

 \bigskip

\subsection{$SU(1,1)$ NLFT on $\R$.}

Suppose $q\in L^1_{\rm loc}(\R)$. For every $T>0$, we let $q^{\langle T\rangle}:=q\cdot \chi_{|x|<T}$.
Define $X(x,k,q^{\langle T\rangle})$ as the solution to
\[
\quad \partial_x{X}=\left(
\begin{array}{cc}
ik & \overline{q}^{\langle T\rangle}\\
q^{\langle T\rangle}&0
\end{array}
\right)X \quad \text{for }\,\, x\in \R \quad \text{and}\quad X(x,k,q^{\langle T\rangle})=\left(\begin{array}{cc}e^{ikx}&0\\0&1\end{array}\right) \quad \text{for}\quad  x\le -T\,,
\]
where $k\in \C$. It is easy to check  that $X\in U(1,1)$  for $k\in \R$, $X$  is $J$-contraction for $k\in \C^+$, and $J$-expansion for $k\in \C^-$.  If $X=:\left(\begin{smallmatrix}
e^{ikx} & 0\\
0 & 1
\end{smallmatrix}\right)\widetilde X$, then
\[
 \partial_x{\widetilde X}=\left(
\begin{array}{cc}
 0& \overline{q}^{\langle T\rangle}e^{-ikx}\\
q^{\langle T\rangle}e^{ikx}&0
\end{array}
\right)\widetilde X \quad \text{for} \quad x\in \R \quad\text{and} \quad  \widetilde X(x,k,q^{\langle T\rangle})=\left(\begin{array}{cc}1&0\\0&1\end{array}\right) \quad \text{for }\quad x\le -T\,.
\]
Now, $\widetilde X\in SU(1,1)$ for $k\in \R$.  For $x\ge T$, we have
\[
\widetilde X(x,k,q^{\langle T\rangle})=:\left(\begin{array}{cc}
\frak{a}(k,q^{\langle T\rangle})& \frak{b}^\#(k,q^{\langle T\rangle})\\
\frak{b}(k,q^{\langle T\rangle})& \frak{a}^\#(k,q^{\langle T\rangle})
\end{array}\right)\,,
\]
where we denote $f^\#(k):=\overline{f(\bar{k})}$ for $k\in \C$.
\smallskip

\noindent {\bf Definition.} {\it For $q^{\langle T\rangle}$, define $SU(1,1)$ NLFT on $\R$ as the map 
\[
q^{\langle T\rangle}\mapsto \overbrace{q^{\langle T\rangle}}:=\left(\begin{array}{c}
\frak{a}(k,q^{\langle T\rangle})\\
\frak{b}(k,q^{\langle T\rangle}) 
\end{array}\right)\,.
\]}
\smallskip
In analogy with $\Z$ setup, we have two questions for the critical case when $q\in L^2(\R)$ (they both have positive answers for $q\in L^p(\R), p\in [1,2)$, see \cite{ck},\cite{kov1},\cite{musc},\cite{st}):\medskip

\noindent{\bf Q1$_\R$ (the strong version of $SU(1,1)$ pointwise convergence of NLFT on} $\R$):  {\it Let $q\in L^2(\R)$. Is it true that the limits
$\lim_{T\to\infty}\frak{a}(k,q^{\langle T\rangle})$ and  $\lim_{T\to\infty}\frak{b}(k,q^{\langle T\rangle})$ exist for a.e. $k\in \R$?}\medskip

and\medskip

\noindent{\bf Q2$_\R$ (a weak version of $SU(1,1)$ pointwise convergence of NLFT on} $\R$):  {\it Let $q\in L^2(\R)$. Is it true that the limit
 $\lim_{T\to\infty}\frak{b}(k,q^{\langle T\rangle})/\frak{a}^\#(k,F^{\langle T\rangle})$ exists for a.e. $k\in \R$?}\medskip

 The first question is motivated by the scattering theory in Schr\"odinger evolution \cite{dm}. Our next result is 
 
 \begin{theorem}\label{t4} The answer to {\bf Q1}$_\R$ is negative. In fact, there is $q\in L^2(\R)$ for which the limits $\lim_{T\to\infty}\frak{a}(k,q^{\langle T\rangle})$ and  $\lim_{T\to\infty}\frak{b}(k,q^{\langle T\rangle})$ do not exist for all $k\in \R$.
 \end{theorem}

 The structure of the paper is as follows. The second section contains the proof of Theorem~\ref{t1}.  There, we also  briefly mention how the proof of Theorem~\ref{t1} can be modified to get Theorem~\ref{t4}.   In the third section, we discuss the connection of $SU(1,1)$ NLFT on $\Z$ to the OPUC theory and prove Theorem~\ref{t2} and Theorem~\ref{t39}. We also explain how the general theorems on convergence of orthogonal series provide the best known results on the pointwise convergence of $SU(1,1)$ NLFT on $\Z$.\medskip
 
 \noindent {\bf Acknowledgments.} I thank Roman Bessonov for helpful discussions.\bigskip\bigskip
 
\noindent{ \bf Notation.} \medskip

$\bullet$ We let $\N=\{1,2,\ldots\}$ and $\Z^+=\{0,1,2\ldots\}$.
$\T$ and $\D$ denote the unit circle  and open unit disc in $\C$, respectively, both centered at zero.\smallskip

$\bullet$ The symbol $C$ denotes the absolute constant which can change the value from formula to formula. If we write, e.g., $C(\alpha)$, this defines a positive function of parameter $\alpha$. \smallskip

$\bullet$ For two non-negative functions $f_1$ and $f_2$, we write $f_1\lesssim f_2$ if  there is an absolute
constant $C$ such that $f_1\le Cf_2$ for all values of the arguments of $f_1$ and $f_2$. We define $\gtrsim$
similarly and say that $f_1\sim f_2$ if $f_1\lesssim f_2$ and
$f_2\lesssim f_1$ simultaneously. If $|f_3|\lesssim f_4$, we will write $f_3=O(f_4)$. If $\alpha$ is a parameter, we write $f_1\le_\alpha f_2$ if $f_1\le C(\alpha)f_2$. The symbol $o_t(1), t\to\infty$ indicates a quantity that converges to zero when $t\to\infty$.\smallskip

$\bullet$ If $E$ is a set, the symbol $\chi_E$ denotes its characteristic function and $|E|$ indicates its Lebesgue measure.\smallskip

$\bullet$ The class $\Sch(\R)$ is the class of Schwartz functions on $\R$, $C_c^\infty(\R)$ is the class of infinitely smooth functions on $\R$ with compact support, and $C^\infty(\T)$ is the class of infinitely smooth functions on $\T$. The symbol $C(\T)$ denotes the space of continuous functions on $\T$. The class $\Sch(\Z)=\{F=\{F_n\}|_{n\in \Z}: |F_n|\le_{\ell} (1+|n|)^{-\ell}, \forall n\in \Z, \ell\in \Z^+\}$ is the Schwartz class on $\Z$. We write $\Sch^*(\Z):=\{F\in \Sch(\Z):  \|F\|_{\ell^\infty(\Z)}<1\}$.  The symbol $A(\D)$ denotes the class of functions analytic in $\D$ and continuous in $\overline{\D}$.\smallskip

$\bullet$ We write $f^{(*)}(z):=\overline{f(\bar{z}^{-1})}$ and $f^\#(k):=\overline{f(\bar{k})}$. For any polynomial $Q$, we let $Q^*:=z^n\overline{Q(\bar{z}^{-1})}, n\in \Z^+$. \smallskip

$\bullet$ For a sequence $F=\{F_n\}|_{n\in \Z}$ defined on $\Z$, $N\in \Z^+$, and $M\in \Z$, we write $F^{\langle N\rangle}=F\cdot \chi_{|n|\le N}$ for truncation and $F_{\to M}: (F_{\to M})_n=F_{n-M}$ for translation. Also, we denote $F^{\le M}:=F\cdot \chi_{n\le M}$.
\smallskip

$\bullet$ The symbol $m$ often stands for the normalized probability Lebesgue measure on $\T$. For  $f\in L^1(\T)$, we define the Fourier transform by $\widehat f(n):=\int_{\T}f(x)\exp(-inx)dm, n\in \Z$.\smallskip

$\bullet$ In the text below, $\nu, n^*\in \N$ will be taken as  large parameters and $\delta,\epsilon$ will be small positive parameters. 
\bigskip

\section{Proofs of Theorem \ref{t1} and Theorem \ref{t4}.}

Suppose $w: \Z\to [1,\infty)$ satisfies
\begin{eqnarray*}
w(-n)=w(n),\\
w(n_1+n_2)\le w(n_1)w(n_2),\\
|n|^{-1}\log w(n)\to 0, \quad |n|\to\infty\,.
\end{eqnarray*}
We call these  $w$  strong Beurling weights \cite{bs}, p.306. The standard example is $w=(1+|n|)^{\beta}, \beta\ge 0$. One says $f\in A_w(\T)$ if
\[
\|f\|_{A_w(\T)}:=\|\widehat f\|_{\ell^1_w(\Z)}:=\sum_{n\in \Z}w(n)|\widehat f(n)|<\infty\,.
\]
In the case when $F\in \ell^1_w(\Z)$ and $w$ is a Beurling weight, we have $\frak{b}(z,F^{\langle N\rangle})\to \frak{b}(z,F)$ and $\frak{a}(z,F^{\langle N\rangle})\to \frak{a}(z,F)$
 for $N\to\infty$ and the convergence is understood in the Wiener's algebra $A_w(\T)$. That allows us to define the NLFT for $F\in \Sch^*(\Z)$ if we choose $w=(1+|n|)^\beta, \forall \beta\ge 0$. \medskip
 
\noindent{\bf Definition.} For $F\in \Sch^*(\Z)$ the NLFT is the map \[
F\mapsto \overbrace F:=\left(\begin{array}{c}
\frak{a}(z,F)\\
\frak{b}(z,F) 
\end{array}\right)\,.
\]
If $F\in \Sch^*(\Z)$, the function $\frak{a}^{(*)}(z,F)\in C^\infty(\T)$ satisfies the following properties:\smallskip

(A) We have $|\frak{a}^{(*)}(z,F)|^2=1+|\frak{b}(z,F)|^2$ for $z\in \T$.\smallskip

(B) $\frak{a}^{(*)}(z,F)\in A(\D)$ is outer function in $\D$ and 
\begin{equation}\label{out1}
\frak{a}^{(*)}(z,F)=\exp\left(   \int_{\T}\frac{\xi+z}{\xi-z} \log|\frak{a}^{(*)}(\xi,F)|dm \right)=\exp\left( \frac 12  \int_{\T}\frac{\xi+z}{\xi-z} \log(1+|\frak{b}(\xi,F)|^2)dm \right), 
\end{equation}
for $z\in \D$.
\smallskip

(C) (nonlinear Plancherel identity \cite{tt,kov})
\begin{equation}\label{sumrule}
-\sum_{n\in \Z}\log(1-|F_n|^2)=2\int_{\T}\log |\frak{a}^{(*)}(\xi,F)|dm=\int_{\T}\log (1+|\frak{b}(\xi,F)|^2)dm\,.
\end{equation}

We need the following two results (see \cite{AMT24}, p. 314 in \cite{bs}, and \cite{gev} for $SU(2)$ case, the same proof works for $SU(1,1)$). The first one is important for solving the Ablowitz-Ladik equation (\cite{al,al1}) via the inverse scattering approach.
\begin{theorem}\label{ht1}The map $F\mapsto \frak{b}(z,F)$ is bijective from $\Sch^*(\Z)$ to $C^\infty(\T)$.\label{t11}
\end{theorem}

\begin{theorem}\label{invt} There is $\tau_0>0$ such that $\|F\|_{\ell_1(\Z)}\le \tau_0 \Rightarrow \|\frak{b}(\cdot,F)\|_{A_1(\T)}\lesssim \|F\|_{\ell_1(\Z)}$ and $\|\frak{b}\|_{A_1(\T)}\le \tau_0 \Rightarrow \|F(\frak{b})\|_{\ell_1(\Z)}\lesssim \|\frak{b}\|_{A_1(\T)}$.
\end{theorem}

We will be using the following notation: given $N,M\in \Z$ and a sequence $F=\{F_n\}|_{n\in \Z}$, we write
\[
F_{\to N}: (F_{\to N})_n=F_{n-N}, \forall n\in \Z\,, F^{\le M}:=F\cdot \chi_{n\le M}\,.
\]
One has the following lemma.
\begin{lemma}
Suppose $F$ and $G$ are compactly supported sequences on $\Z$, $\supp F\subset (-\infty,N_1]$, $\supp G\subset [-N_2,\infty)$, and $N>N_1+N_2$. Then, for each $M\ge -N_2$, we have
\begin{eqnarray}\label{mul13}
\quad\quad\frak{a}^{(*)}(z,F+(G^{\le M})_{\to N})=\frak{a}^{(*)}(z,F)\frak{a}^{(*)}(z,G^{\le M})+z^{N}\frak{b}^{(*)}(z,F)\frak{b}(z,G^{\le M}),\\
\frak{b}(z,F+(G^{\le M})_{\to N})=z^{N}\frak{b}(z,G^{\le M})\frak{a}(z,F)+\frak{b}(z,F)\frak{a}^{(*)}(z,G^{\le N}),\label{mul23}\\
|\frak{r}(z,F+(G^{\le M})_{\to N})-\frak{r}(z,F)|\le{|\frak{r}(z,G^{\le M})|}/({1-|\frak{r}(z,G^{\le M})|}),\, \quad z\in \T\,.\label{mul33}
\end{eqnarray}
\end{lemma}
\begin{proof}
We start with the first observation  (p.6, \cite{tt}) that for each $H$ with compact support, we have
\begin{equation}\label{obbs}
\frak{a}^{(*)}(z,H_{\to M})=\frak{a}^{(*)}(z,H),\quad
\frak{b}(z,H_{\to M})=\frak{b}(z,H)z^{M}, \quad \forall M\in \Z\,.
\end{equation}
Next, notice that $\supp F\cap \supp ((G^{\le M})_{\to N})=\emptyset$ so \eqref{rec5} and \eqref{obbs} give the first two identities in the lemma. Dividing the second by the first, one has
\[
\frak{r}(z,F+(G^{\le M})_{\to N})=\frac{\frak{r}(z,F)+z^{N}\frak{r}(z,G^{\le M})\frak{a}(z,F)/\frak{a}^{(*)}(z,F)}{1+z^{N}\frak{r}(z,G^{\le M})\frak{b}^{(*)}(z,F)/\frak{a}^{(*)}(z,F)}\,,
\]
so
\begin{eqnarray*}
\frak{r}(z,F+(G^{\le M})_{\to N})-\frak{r}(z,F)=
z^{N}\frak{r}(z,G^{\le M})\frac{\frak{a}(z,F)/\frak{a}^{(*)}(z,F)-\frak{r}(z,F)\frak{b}^{(*)}(z,F)/\frak{a}^{(*)}(z,F)}{1+z^{N}\frak{r}(z,G^{\le M})\frak{b}^{(*)}(z,F)/\frak{a}^{(*)}(z,F)}\\
\stackrel{\eqref{sul}}{=}
\frac{z^{N}\frak{r}(z,G^{\le M})}{(\frak{a}^{(*)}(z,F))^2}\frac{1}{1+z^{N}\frak{r}(z,G^{\le M})\frak{b}^{(*)}(z,F)/\frak{a}^{(*)}(z,F)}
\end{eqnarray*}
and, recalling that $|\frak{a}^{(*)}(z,F)|\ge 1$, $|\frak{r}(z,F)|<1$, and $|\frak{b}^{(*)}(z,F)/\frak{a}^{(*)}(z,F)|<1$ for $z\in \T$, we get
\begin{eqnarray*}
|\frak{r}(z,F+(G^{\le M})_{\to N})-\frak{r}(z,F)|\le{|\frak{r}(z,G^{\le M})|}/{(1-|\frak{r}(z,G^{\le M})|)},\, \quad z\in \T\,.
\end{eqnarray*}
\end{proof}

\begin{lemma} \label{ton1} Suppose $F,G\in \Sch^*(\Z)$ and $N\in \N$. Then, as $N\to\infty$
\begin{eqnarray}\label{f1}
\frak{a}^{(*)}(z,F^{\langle N\rangle}+(G^{\langle N\rangle})_{\to 3N})=\frak{a}^{(*)}(z,F)\frak{a}^{(*)}(z,G)+z^{3N}\frak{b}^{(*)}(z,F)\frak{b}(z,G)+o(1), \\
\frak{b}(z,F^{\langle N\rangle}+(G^{\langle N\rangle})_{\to 3N})=z^{3N}\frak{b}(z,G)\frak{a}(z,F)+\frak{b}(z,F)\frak{a}^{(*)}(z,G)+o(1)\label{f2},\\
\limsup_{N\to\infty}|\frak{r}(z,F^{\langle N\rangle}+(G^{\langle N\rangle})_{\to 3N})-\frak{r}(z,F^{\langle N\rangle})|\le{|\frak{r}(z,G)|}/{(1-|\frak{r}(z,G)|)}\label{last1}
\end{eqnarray}
uniformly in $z\in \T$.
\end{lemma}
\begin{proof}
Given $N$, consider truncated coefficients $F^{\langle N\rangle}$ and $G^{\langle N\rangle}$ and apply the previous lemma to get
\begin{eqnarray}\label{mul1}
\quad\quad\frak{a}^{(*)}(z,F^{\langle N\rangle}+(G^{\langle N\rangle})_{\to 3N})=\frak{a}^{(*)}(z,F^{\langle N\rangle})\frak{a}^{(*)}(z,G^{\langle N\rangle})+z^{3N}\frak{b}^{(*)}(z,F^{\langle N\rangle})\frak{b}(z,G^{\langle N\rangle}),\\
\frak{b}(z,F^{\langle N\rangle}+(G^{\langle N\rangle})_{\to 3N})=z^{3N}\frak{b}(z,G^{\langle N\rangle})\frak{a}(z,F^{\langle N\rangle})+\frak{b}(z,F^{\langle N\rangle})\frak{a}^{(*)}(z,G^{\langle N\rangle}),\label{mul2}\\
|\frak{r}(z,F^{\langle N\rangle}+(G^{\langle N\rangle})_{\to 3N})-\frak{r}(z,F^{\langle N\rangle})|\le{|\frak{r}(z,G^{\langle N\rangle})|}/{(1-|\frak{r}(z,G^{\langle N\rangle})|)}\,.
\end{eqnarray}
Since $F,G\in \Sch^*(\Z)$, we have for $N\to\infty$
\begin{eqnarray*}
\frak{a}(z,G^{\langle N\rangle})=\frak{a}(z,G)+o(1),\quad \frak{a}(z,F^{\langle N\rangle})=\frak{a}(z,F)+o(1),\\
\frak{b}(z,G^{\langle N\rangle})=\frak{b}(z,G)+o(1),\quad \frak{b}(z,F^{\langle N\rangle})=\frak{b}(z,F)+o(1)
\end{eqnarray*}
uniformly in $z\in \T$. Substituting these identities into the previous formulas gives the statement of the lemma.
\end{proof}

Given $F\in \Sch^*(\Z)$, the function $a^{(*)}(z,F)\in A(\D)$ is outer in $\D$, $a^{(*)}(z)\in C^\infty(\T)$, and the formula \eqref{out1} holds. So, the argument of $a^{(*)}(z,F)$ is well-defined by the formula
\[
\arg \frak{a}^{(*)}(z,F)=\int_{\T}\Im\left(\frac{\xi+z}{\xi-z}\right) \log|\frak{a}^{(*)}(\xi,F)|dm = \frac 12   \int_{\T}\Im \left(\frac{\xi+z}{\xi-z}\right) \log(1+|\frak{b}(\xi,F)|^2)dm\,.
\]
These formulas extend to $z\in \T$ where the integrals are represented by the Hilbert transform
\begin{equation}\label{hilb}
\arg \frak{a}^{(*)}(e^{i\phi},F)=\frac{1}{4\pi}\int_{[0,2\pi)} \cot\left(\frac{\phi-\theta}{2}\right) \cdot \log(1+|\frak{b}(e^{i\theta},F)|^2)d\theta\,,
\end{equation}
which is understood in v.p. sense. \smallskip

\noindent {\bf Remark.} To visualize, consider a straight segment $\{re^{i\phi}, r\in [0,1]\}$ connecting the origin to a point $e^{i\phi}$ on $\T$. The corresponding image, i.e., the curve $\{\frak{a}^{(*)}(re^{i\phi},F), r\in [0,1]\}$, lies outside $\D$ and it connects a point $\frak{a}^{(*)}(0,F)=\prod_{n\in \Z}(1-|F_n|^2)^{-\tfrac 12}\in [1,\infty)$ to a point $\frak{a}^{(*)}(e^{i\phi},F)$. The total variation of the argument of a point on that curve is  equal to $\arg \frak{a}^{(*)}(e^{i\phi},F)$, the quantity we are interested in.\smallskip

First, we address the following variational problem for the maximal function. Suppose $I\subset \T, I\neq \T$ is an arc. For every positive $\omega$, we define
\[
\frak{F}_{I,\omega}:=\sup_{F\in \Sch^*(\Z), \|F\|_{\ell^2(\Z)}\le \omega}\inf_{z\in I} \sup_{N\in \Z^+}|\arg a^{(*)}(z,F^{\langle N\rangle})|\,.
\]
Clearly, $\frak{F}_{I,\omega_1}\le \frak{F}_{I,\omega_2}$ if $\omega_1\le \omega_2$ and $\frak{F}_{I_1,\omega}\ge \frak{F}_{I_2,\omega}$ if $I_1\subset I_2$.

\begin{theorem}For every $I$ and $\omega$, we have \,\,$\frak{F}_{I,\omega}=+\infty$.\label{gothg}
\end{theorem}
\begin{proof}Without loss of generality, we take $I=\{e^{i\phi}, \phi\in [\pi/2,3\pi/2]\}$. Then, we let $\nu\in \N$ be a large parameter and let $\phi_j=2\pi j/\nu, j\in \{0,\ldots,\nu-1\}$ and $\Delta_j=[\phi_j,\phi_{j+1})$ be the corresponding intervals that provide a partition of $[0,2\pi)$. Let $\rho\in C^\infty(\R)$ be a nonnegative function with $\supp\rho=[0,2\pi]$. Let $\delta$ be a small positive parameter. Define $b_j(e^{i\phi})=\delta\rho(\nu(\phi-\phi_j))$. We clearly have that
\begin{eqnarray}
\supp b_j=\{e^{i\phi}, \phi\in \Delta_j\},\\
\label{p1} b_{j_1}b_{j_2}=0, \, j_1\neq j_2,\\
\label{pok1} |b_{j}|+\ldots+|b_{j+s}|\lesssim \delta, \forall j,s\,.
\end{eqnarray}
Since $b_j\in C^\infty(\T)$, it uniquely defines $q^{(j)}\in \Sch^*(\Z)$ such that $b_j=\frak{b}(z,q^{(j)})$ by Theorem \ref{ht1}. From \eqref{sumrule}, we know that 
\begin{equation}\label{lok1}
\|q^{(j)}\|_{\ell^2(\Z)}^2\sim \delta^2/\nu\,.
\end{equation}
We also have
$
\|b_j\|_{A_1(\T)}\lesssim \delta
$
which, by Theorem \ref{invt}, implies 
\begin{equation}\label{cont9}
\|q^{(j)}\|_{\ell^1(\Z)}\lesssim \delta\,.
\end{equation}
Each $b_j$ defines $\frak{a}_{j}^{(*)}(z)$ by the formula \eqref{out1} and $\frak{a}_{j}^{(*)}(z)=\frak{a}_{j}^{(*)}(z,q^{(j)})$. We introduce the function 
$\frak{A}_{j}(z)=\frak{a}_0^{(*)}(z)\cdot \ldots\cdot \frak{a}_{j}^{(*)}(z), \, j\in \{0,\ldots,\kappa-1\}\,.
$

\begin{lemma}$\frak{A}_{j}$ satisfies the following properties:

\noindent {\rm (A)} $\frak{A}_{j}\in A(\D)$, it is outer and 
\begin{equation}\label{b1}
  |\frak{A}_{j}(z)|\quad \left\{
  \begin{array}{ll}
  \le 1+C\delta^2, &z\in \overline{\D},\\
  =1, &z=e^{i\phi}, \phi\notin \Delta_0\cup\ldots\cup\Delta_j\,.
  \end{array}
  \right.
\end{equation}
{\rm (B)} For the argument of $\frak{A}_j$, we get an estimate:
\begin{equation}\label{lb1}
|\arg \frak{A}_j(e^{i\phi})|\ge \delta^2(C_1\log\nu-C_2), \quad \phi\in \Delta_{j+1}, \quad j\in [\nu/10,9\nu/10]\,.
\end{equation}
\end{lemma}
\begin{proof}The product of outer functions is an outer function and $A(\D)$ is an algebra. Also, from \eqref{p1}, an identity $|\frak{a}(e^{i\phi},q^{(j)})|^2=1+|\frak{b}(e^{i\phi},q^{(j)})|^2$, and the choice of $b_j$, we get \eqref{b1} when $z\in \T$. The application of the  maximum principle extends the bound to $z\in \D$.\vspace{1cm}

\begin{tikzpicture}
    \draw[thick, black] (-0.1,0) -- (14,0);
    \draw[black, fill=black] (1,0) circle (0.4mm);
    \node at (1,-0.4) {$\phi_{j-3}$};
    \draw[black, fill=black] (2,0) circle (0.4mm);
    \node at (2,-0.4) {$\phi_{j-2}$};
    \draw[black, fill=black] (3,0) circle (0.4mm);
    \node at (3,-0.4) {$\phi_{j-1}$};
    \draw[black, fill=black] (4,0) circle (0.4mm);
    \node at (4,-0.4) {$\phi_{j}$};
    \draw[black, fill=black] (5,0) circle (0.4mm);
    \node at (5,-0.4) {$\phi_{j+1}$};
    \node at (6,0.2) {$\phi$};
    \draw[black, fill=black] (5.8,0) circle (0.4mm);
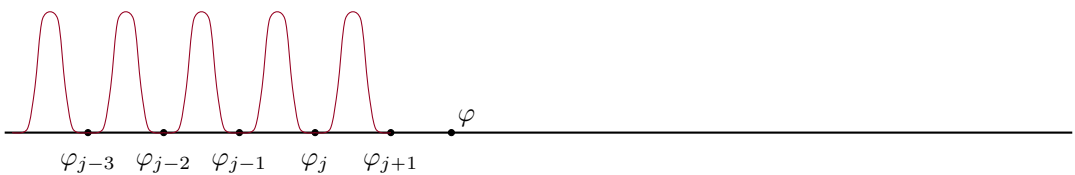
    \node at (7,-1) {Figure 1: $\log |\frak{A}_j|$. Creation of logarithmic growth by piling bumps to the left of $\phi$.};
     \node at (7,-1.5) { The ``height'' of each ``petal'' is $\sim \delta^2$ and its ``width'' is $\sim 1/\nu  $. We have $\phi_s-\phi_{s-1}=(2\pi)/\nu  $\,.};

    \foreach \r in {0,...,4}
    {
    \draw[darkbrown, samples=1000, domain=-1:1, smooth] plot [smooth, solid] coordinates {(\r,0.0) (\r+0.2,0.05)   (\r+0.3, 0.6) 
    (\r+0.4,1.5)  (\r+0.6,1.5)  (\r+0.7,0.6)  (\r+0.8,0.05) (\r+1,0.0)};
    }

    \end{tikzpicture}
\vspace{0.5cm}

The estimate \eqref{lb1} follows from the estimate on the Hilbert transform:
\[
\arg \frak{A}_j(e^{i\phi})\stackrel{\eqref{hilb}}{=}\frac{1}{4\pi}\int_{[0,2\pi)} \cot\left(\frac{\phi-\theta}{2}\right) \cdot \left( \sum_{s=0}^{j}\log(1+|b_s(e^{i\theta})|^2)\right)d\theta
\]
after we notice that
\[
\sum_{s=0}^{j}\log(1+|b_s(e^{i\theta})|^2)\quad
\left\{
\begin{array}{ll}
\ge 0, &\theta\in [0,2\pi),\\
\le C\delta^2,& \theta\in [0,2\pi),\\
=0,& \theta\in [\phi_{j+1},2\pi),\\
\sim \delta^2,& \theta\in \cup_{0\le\ell\le j} [\phi_\ell+0.1\nu^{-1}, \phi_{\ell+1}-0.1\nu^{-1}]
\end{array}
\right.
\]
so
\begin{eqnarray*}
\left|\int_{[0,2\pi)} \cot\left(\frac{\phi-\theta}{2}\right) \cdot \left( \sum_{s=0}^{j}\log(1+|b_s(e^{i\theta})|^2)\right)d\theta\right|\ge \hspace{5cm}\\
C_1\delta^2\left|\sum_{\ell=j-\lceil\nu/20\rceil}^{j}\int_{ \phi_\ell+0.1\nu^{-1}}^{\phi_{\ell+1}-0.1\nu^{-1}} \cot\left(\frac{\phi-\theta}{2}\right) d\theta\right|-C_2\delta^2\ge C_1 \delta^2\log\nu-C_2\delta^2\,,
\end{eqnarray*}
when $\phi\in \Delta_{j+1}$ and $ j\in [\nu/10,9\nu/10]$.
\end{proof}


We now recursively produce the set $\{F^{(0)},\ldots,F^{(\nu-1)}\}$ where each $F^{(j)}, j\in \{1,\ldots,\nu-1\}$ has compact support as follows: $F^{(0)}=q^{(0)}$ and
\begin{eqnarray} 
F^{(j)}=(F^{(j-1)})^{\langle T_j\rangle}+((q^{(j)})^{\langle T_j\rangle })_{\to 3T_j}, \, j\in \{1,\ldots,\nu-1\}\,,\label{rsd}
\end{eqnarray}
where $T_j$ is a large natural numbers and $10T_j<T_{j+1}, \forall j$. \smallskip

Notice that for such choice of $\{T_j\}$, we have  $(F^{(j-1)})^{\langle T_j\rangle}=F^{(j-1)}, j\in \{2,\ldots,\nu-1\}$. For shorthand, we will write
\[
\limsup_{\vec{T}_j\to\infty}G(T_1,\ldots,T_j)=\limsup_{T_1\to \infty}\,(\ldots\limsup_{T_j\to\infty}G(T_1,\ldots,T_{j})\ldots)
\]
for any function $G(T_1,\ldots,T_j)$.
\begin{lemma} For $j\in \{1,\ldots,\nu-1\}$, we have
\begin{eqnarray}\label{fi1}
\|F^{(j)}\|^2_{\ell^2(\Z)}\le C\delta^2j/\nu\,,\\
\limsup_{\vec{T}_j\to\infty}|\frak{a}^{(*)}(z,F^{(j)})-\frak{A}_{j}(z)|=0 \,\, \text{uniformly in} \,z\in \overline{\D}\,,   \label{t3}\\
\limsup_{\vec{T}_j\to\infty}|\frak{b}(z,F^{(j)}|\le |\frak{A}_{j}(z)|(|\frak{b}(z,q^{(0)})|+\ldots+|\frak{b}(z,q^{(j)})|)\,\, \text{uniformly in} \,z\in {\T},\label{pr2}\\
\limsup_{{T}_j\to\infty}|\frak{r}(z,F^{(j)})-\frak{r}(z,F^{(j-1)})|\le {|\frak{r}(z,q^{(j)})|}/({1-|\frak{r}(z,q^{(j)})|})  \,\, \text{uniformly in} \,z\in {\T}\label{pr28}\,.
\end{eqnarray}
\end{lemma}
\begin{proof}
The bound \eqref{fi1} is immediate from \eqref{lok1} and the construction.
To prove \eqref{t3} and \eqref{pr2}, we proceed by induction. For $j=1$, they follows from \eqref{f1}, \eqref{f2}, \eqref{p1}, and \eqref{pok1}. Suppose the claims hold for $j-1$. We use \eqref{f1} to write (with fixed $T_1,\ldots,T_{j-1}$)
\begin{equation}\label{kun1}
\frak{a}^{(*)}(z,F^{(j)})=\frak{a}^{(*)}(z,F^{(j-1)})\frak{a}^{(*)}(z,q^{(j)})+z^{3T_j}\frak{b}^{(*)}(z,F^{(j-1)})\frak{b}(z,q^{(j)})+o_{T_j}(1)
\end{equation}
and
\begin{eqnarray*}
\limsup_{\vec{T}_j\to\infty}|\frak{a}^{(*)}(z,F^{(j)})-\frak{A}_{j}(z)|\le \Bigl(\limsup_{\vec{T}_j\to\infty}|\frak{a}^{(*)}(z,F^{(j-1)})-\frak{A}_{j-1}(z)|\Bigr) |\frak{a}^{(*)}(z,q^{(j)})|\\
+\Bigl(\limsup_{\vec{T}_j\to\infty}|\frak{b}^{(*)}(z,F^{(j-1)})|\Bigr)\cdot |\frak{b}(z,q^{(j)})|+\limsup_{\vec{T}_j\to\infty}|o_{T_j}(1)|\,.
\end{eqnarray*}
The last term is zero. In the first and the second terms, the functions involved are independent of $T_j$ so we can use inductive assumptions for the $j-1$ and \eqref{p1} to obtain 
 \eqref{t3} uniformly in $z\in \T$. By the maximum principle, we get our statement in $\D$, as well. 
By \eqref{f2},
\begin{equation}\label{kun2}
\frak{b}(z,F^{(j)})=z^{3T_j}\frak{b}(z,q^{(j)})\frak{a}(z,F^{(j-1)})+\frak{b}(z,F^{(j-1)})\frak{a}^{(*)}(z,q^{(j)})+o_{T_j}(1)
\end{equation}
uniformly in $z\in \T$. We argue similarly,
\begin{eqnarray*}
\limsup_{\vec{T}_j\to\infty}|\frak{b}(z,F^{(j)})|\le \Bigl(\limsup_{\vec{T}_j\to\infty}|\frak{a}(z,F^{(j-1)})|\Bigr)|\frak{b}(z,q^{(j)})|+\hspace{4cm}\\
\Bigr(\limsup_{\vec{T}_j\to\infty}|\frak{b}(z,F^{(j-1)})|\Bigl)|\frak{a}^{(*)}(z,q^{(j)})|+\limsup_{\vec{T}_j\to\infty}o_{T_j}(1)\,.
\end{eqnarray*}
The last term is zero and the first two are independent of $T_j$ so applying both inductive assumptions for $j-1$ we get our statement.
 The bound \eqref{pr28} follows from \eqref{last1}.
\end{proof}
Besides a large parameter $\nu$ and small parameter $\delta>0$, we introduce an additional small  parameter $\epsilon>0$.\smallskip

\noindent {\bf Definition.} {\it Given $(\nu,\delta,\epsilon)$, we call $F_{\nu,\delta,\epsilon}:=F^{(\nu-1)}_{\nu,\delta,\epsilon}$ from \eqref{rsd}
 an { $(\nu,\delta,\epsilon)$-daisy} if for every $j\in \{1,\ldots,\nu-1\}$ we have
\begin{eqnarray}\label{fi2q}
\|F_{\nu,\delta,\epsilon}\|^2_{\ell^2(\Z)}\le C\delta^2 \,,\\
|\frak{a}^{(*)}(z,F_{\nu,\delta,\epsilon}^{(j)})-\frak{A}_{j}(z)|\le \epsilon \,\, \quad \text{uniformly in} \,z\in \overline{\D}\,,   \label{t3q}\\
|\frak{b}(z,F_{\nu,\delta,\epsilon}^{(j)})|\lesssim (|\frak{b}(z,q^{(0)})|+\ldots+|\frak{b}(z,q^{(j)})|)+\epsilon\lesssim \delta+\epsilon\,\,\quad \text{uniformly in} \,z\in {\T}\label{pr2q}\,,\\
\supp F_{\nu,\delta,\epsilon}^{(j)}\subset [-T_1,4T_j]\,,\\ \label{cute}
F_{\nu,\delta,\epsilon}^{\le 4T_j}=F_{\nu,\delta,\epsilon}^{\langle 4T_j\rangle }=F_{\nu,\delta,\epsilon}^{(j)}\,,\label{cut1}\\
\sup_{t\in \Z}|\frak{r}(z,F_{\nu,\delta,\epsilon}^{\le t})|\lesssim\delta+\epsilon\quad \text{uniformly in} \,z\in {\T}\,.\label{cut2}
\end{eqnarray}}
\begin{lemma}
One can choose the parameters $T_1,\ldots,T_{\nu-1}$ large enough such that the $(\nu,\delta,\epsilon)$-daisy exists for every  triple 
$(\nu,\delta,\epsilon)$. 
\end{lemma}
\begin{proof}
Application of the previous lemma guarantees that \eqref{fi2q}-\eqref{cut1} can be satisfied. Now, we focus on \eqref{cut2}. First, we notice that \eqref{t3q} and \eqref{pr2q} yield
\begin{equation}
|\frak{r}(z,F_{\nu,\delta,\epsilon}^{(j)}|\lesssim \delta+\epsilon\,\,\quad \text{uniformly in} \,z\in {\T}\label{pr2q88}
\end{equation}
for each $j$.  For each $t\in \Z$, define $j(t): \min_{s\in \{1,\ldots,\nu-1\}}|t-4T_s|=|t-4T_{j(t)}|$. Then,
\[
\sum_{t\le s\le 4T_{j(t)}}|(F_{\nu,\delta,\epsilon})_s|\stackrel{\eqref{cont9}}{\lesssim} \delta \quad \left( \text{or} \sum_{4T_{j(t)}\le s\le t}
|(F_{\nu,\delta,\epsilon})_s|\lesssim\delta\right)
\]
and
\[
|\frak{r}(z,F_{\nu,\delta,\epsilon}^{\le t})|\le |\frak{r}(z,F_{\nu,\delta,\epsilon}^{\le t})-\frak{r}(z,F_{\nu,\delta,\epsilon}^{(j(t))})|+|\frak{r}(z,F_{\nu,\delta,\epsilon}^{(j(t))})|\stackrel{\eqref{pr2q88}+\text{Lemma}\, \ref{var}}{\lesssim}\delta+\epsilon
\]
and existence of the daisy is proved.
\end{proof}



Take $\epsilon=0.1$. Then, \eqref{lb1}, \eqref{t3q}, and \eqref{cute} give us 
\begin{equation}\label{lb101}
|\arg \frak{a}^{(*)}(e^{i\phi},F_{\nu,\delta,\epsilon}^{\langle 4T_j\rangle})|\ge  \delta^2(C_1\log\nu-C_2), \quad \phi\in \Delta_{j+1}, \quad j\in [\nu/10,9\nu/10]\,.
\end{equation}
Taking the sequence $\{\nu_s\}$ and $\delta$ such that 
\[
C\delta^2\le \omega^2, \quad \lim_{s\to\infty}\nu_s=+\infty
\]
and the corresponding $(\nu_s,\delta,0.1)$-daisies, the bound \eqref{lb101} yields the  Theorem \ref{gothg} for  $I=\{e^{i\phi}, \phi\in [\pi/2,3\pi/2]\}$. Clearly, our construction can accommodate any other arc. The proof of Theorem \ref{gothg} is finished.
\end{proof}
Introduce $\ell(\nu,\delta,\epsilon)$ as
\[
\ell(\nu,\delta,\epsilon)=\min \{s: s\in \Z^+, \supp F_{\nu,\delta,\epsilon}\subset [-s,s]\}\,.
\]
Since each daisy is compactly supported, such a number is always well-defined and finite. \medskip

\noindent {\it Proof of Theorem \ref{t1}.} Fix some $n^*\in \N$, a large parameter, and let
\begin{equation}\label{col2}
\nu_n=\lceil\exp(\exp(n^2))\rceil,\delta_n=\exp(-n),\epsilon_n=\exp(-n),
\end{equation} where $n\ge n^*$. Let $D^{(n)}:=F_{\nu_n,\delta_n,\epsilon_n}$, a $(\nu_n,\delta_n,\epsilon_n)$-daisy from the previous lemma. Partition $\T$ to the left semicircle $\T_\ell:=\{e^{i\phi}, \phi\in [\pi/2,3\pi/2]\}$ and the right semicircle $\T_{r}:=\{e^{i\phi}, \phi\in [-\pi/2,\pi/2]\}$. Using our construction, we can arrange (see \eqref{cut1} and \eqref{lb101} for $\T_\ell$) that 
\begin{eqnarray}\label{gr1}
\inf_{z\in \T_{\ell}} \sup_{N\in \N}|\arg \frak{a}^{(*)}(z,(D^{(n)})^{\le N})|\gtrsim \delta_n^2\log\nu_n
\end{eqnarray}
for even $n$ and 
\begin{eqnarray}
\inf_{z\in \T_{r}} \sup_{N\in \N}|\arg \frak{a}^{(*)}(z,(D^{(n)})^{\le N})|\gtrsim \delta_n^2\log\nu_n
\end{eqnarray}
for odd $n$. Recall that  $\supp D^{(n)}\subset [-\ell(\nu_n,\delta_n,\epsilon_n),\ell(\nu_n,\delta_n,\epsilon_n)]$.  We define 
\[
H^{(n^*)}:=(D^{(n^*)})_{\to L_{n^*}}\,,
\]
where $L_{n^*}>\ell(\nu_{n^*},\delta_{n^*},\epsilon_{n^*})$. Then, we recursively define 
\begin{equation}\label{sm1}
H^{(n)}=H^{(n-1)}+(D^{(n)})_{\to 3L_{n}}, \, n>n^*\,,
\end{equation}
where $L_n$ satisfies $L_n>2L_{n-1}$ and $L_n>\ell(\nu_n,\delta_n,\epsilon_n)$, which ensures
\begin{equation}\label{list}
\supp H^{(n-1)}\subset [1,2L_n), \, \supp D^{(n)}\subset  (-L_n,L_n), \, \supp H^{(n-1)}\cap \supp (D^{(n)})_{\to 3L_{n}}=\emptyset\,.
\end{equation}
Let 
\begin{equation}\label{eyh}
H=\lim_{n\to\infty} H^{(n)}.
\end{equation} Notice that our choice of $\{L_n\}$ guarantees that:\smallskip
\begin{eqnarray}
\text{  $\supp H\subset \N$,}\\
\text{ by \eqref{list}, the terms in \eqref{sm1} have disjoint supports  so $H$ is well defined,}\\
H_n=0\,\, \text{for} \,\,n\in \cup_{s> n^*} [4L_s,2L_{s+1}],\\
\|H\|^2_{\ell^2(\Z)}\le \sum_{n\ge n^*}\|D^{(n)}\|^2_{\ell^2(\Z)}\stackrel{\eqref{fi1}}{\lesssim} \sum_{n\ge n^*}\delta_n^2\stackrel{\eqref{col2}}{\lesssim} e^{-2n^*}\,.
\end{eqnarray}
Next, we will study $\frak{a}(z,H^{(n)})$ and $\frak{b}(z,H^{(n)})$.
\begin{lemma}
For every $n\ge n^*$, we have
\begin{equation}\label{glob1}
 |\frak{a}^{(*)}(z,H^{(n)})|\le 1+C\exp(-n^*), \, |\frak{b}(z,H^{(n)})|\le C\exp(-n^*)
\end{equation}
and
\begin{equation}\label{cauch1}
|\frak{r}(z,H^{(n+1)})-\frak{r}(z,H^{(n)})|\le  {|\frak{r}(z,D^{(n+1)})|}/{(1-|\frak{r}(z,D^{(n+1)})|)}\lesssim \exp(-n)
\end{equation}
uniformly in $z\in \T$ and $n\ge n^*$. 
\end{lemma}
\begin{proof} From \eqref{mul13} and \eqref{mul23}, we get
\begin{eqnarray}\label{per1}
|\frak{a}^{(*)}(z,H^{(n+1)})|\le
|\frak{a}^{(*)}(z,H^{(n)})\frak{a}(z,D^{(n+1)})|+|\frak{b}(z,H^{(n)})\frak{b}(z,D^{(n+1)})|\,,\\\label{per2}
|\frak{b}(z,H^{(n+1)})|\le
 |\frak{a}^{(*)}(z,H^{(n)})\frak{b}(z,D^{(n+1)})|+|\frak{b}(z,H^{(n)})\frak{a}^{(*)}(z,D^{(n+1)})|\,,
\end{eqnarray}
uniformly in $z\in \T$. 
Next, from \eqref{b1}, \eqref{t3q},  and \eqref{pr2q}, we get
\[
|\frak{a}^{(*)}(z,D^{(n)})|\le 1+C\exp(-n),
|\frak{b}(z,D^{(n)})|\lesssim \exp(-n)\,.
\]
Adding \eqref{per1} and \eqref{per2}, we get
\[
m_{n+1}\le m_n(1+C\exp(-n))\,,
\]
where $m_n:=|\frak{a}^{(*)}(z,H^{(n)})|+|\frak{b}(z,H^{(n)})|$. Since $m_{n^*}\le 1+C\exp(-n^*)$, we use Lemma \ref{lemap1} from Appendix to get
\[
m_{n}\le 1+C\exp(-n^*), \quad \forall n\ge n^*\,.
\]
Substituting this estimate back into \eqref{per1} and \eqref{per2}, we have
\begin{eqnarray}\label{per11}
|\frak{a}^{(*)}(z,H^{(n+1)})|\le
 |\frak{a}^{(*)}(z,H^{(n)})(1+C\exp(-n))+C\exp(-n)\,,\\\label{per21}
|\frak{b}(z,H^{(n+1)})|\le
 |\frak{b}(z,H^{(n)})(1+C\exp(-n))+C\exp(-n)\label{per12}
\end{eqnarray}
uniformly in $z\in \T$. Applying Lemma \ref{lemap1} one more time, we get
\begin{equation}
 |\frak{a}^{(*)}(z,H^{(n)})|\le 1+C\exp(-n^*), \, |\frak{b}(z,H^{(n)})|\le C\exp(-n^*)
\end{equation}
uniformly in $z\in \T$ and $n\ge n^*$. The bound \eqref{cauch1} follows from \eqref{mul33} and \eqref{cut2}.
\end{proof}

 We claim that 
\begin{lemma}\label{gr9}For $H$, given by \eqref{eyh} with large enough $n^*$, we get
\[
\sup_{T>0} |\arg \frak{a}^{(*)}(z,H^{\langle T\rangle})|=+\infty
\]
for all $z\in \T$.
\end{lemma}
\begin{proof}The proof is by contradiction. Suppose $z^*\in \T_{\ell}$ is a point (the case $z^*\in \T_{r}$ is  handled similarly) where $|\arg \frak{a}^{(*)}(z^*,H^{\langle T\rangle})|<C^*$ for all $T\in \N$. By \eqref{gr1}, we can take a sequence $\{n_j\}$ and the numbers $\{N_j\}$ such that 
\begin{equation}\label{up1}
|\arg \frak{a}^{(*)}(z^*,(D^{(n_j)})^{\le N_j})|\gtrsim \delta_{n_j}^2\log\nu_{n_j}\stackrel{\eqref{col2}}{\gtrsim }\exp(n_j^2-2n_j)\to +\infty, \quad j\to\infty\,.
\end{equation}
Notice that $H^{(n_j-1)}+((D^{(n_j)})^{\le N_j})_{\to 3L_{n_j}}=H^{\langle 3L_{n_j}+N_j\rangle}=H^{\le (3L_{n_j}+N_j)}$.
From \eqref{mul13}, one has
\begin{eqnarray}
\frak{a}^{(*)}(z,H^{(n_j-1)}+((D^{(n_j)})^{\le N_j})_{\to 3L_{n_j}})=\\
\frak{a}^{(*)}(z,H^{(n_j-1)})\frak{a}^{(*)}(z,(D^{(n_j)})^{\le N_j})+z^{3L_{n_j}}\frak{b}^{(*)}(z,H^{(n_j-1)})\frak{b}(z,(D^{(n_j)})^{\le N_j})\nonumber
\end{eqnarray}
and
\begin{eqnarray}
|\frak{a}^{(*)}(z,H^{\langle 3L_{n_j}+N_j\rangle})-\frak{a}^{(*)}(z,H^{(n_j-1)})\frak{a}^{(*)}(z,(D^{(n_j)})^{\le N_j})\le |\frak{b}^{(*)}(z,H^{(n_j-1)})\frak{b}(z,(D^{(n_j)})^{\le N_j})\nonumber\\
\stackrel{\eqref{pr2q}+\eqref{cut2}+\eqref{glob1}}{\lesssim }\exp(-(n^*+n_j))
\end{eqnarray}
uniformly in $z\in \T$. We can extend this bound to $z\in \D$ by applying the maximum principle.  Taking $n^*$ sufficiently large, we can guarantee that the r.h.s. is smaller than $0.1$ which yields
\[
\arg \frak{a}^{(*)}(z,H^{\langle 3L_{n_j}+N_j\rangle})=\arg \frak{a}^{(*)}(z,H^{(n_j-1)})+\arg \frak{a}^{(*)}(z,(D^{(n_j)})^{\le N_j})+O(\exp(-(n^*+n_j))
\]
for $z\in \T$. The property \eqref{list} shows that $H^{( n_j-1)}=H^{\le 2L_{n_j}}=H^{\langle 2L_{n_j}\rangle}$. Then, for $n^*$ large enough
\begin{eqnarray*}
|\arg \frak{a}^{(*)}(z^*,(D^{(n_j)})^{\le N_j})|\le \hspace{9cm}\\
|\arg \frak{a}^{(*)}(z^*,H^{\langle 3L_{n_j}+N_j\rangle})|+|\arg \frak{a}^{(*)}(z,H^{\langle 2L_{n_j}\rangle})|+O(\exp(-(n^*+n_j))<3C^*
\end{eqnarray*}
and that contradicts \eqref{up1}.
\end{proof}
\begin{lemma} \label{sss}The sequence $\{\frak{r}(z,H^{\le n})\}$ converges uniformly on $\T$. 
\end{lemma}
\begin{proof}
Indeed, by \eqref{cauch1} and Cauchy criterion, the sequence $\{\frak{r}(z,H^{(s)})\}$ converges uniformly on $\T$. For $n$,  define $m_n$ as any minimizer in $\min_{p\ge n^*}|n-3L_p|=|n-3L_{m_n}|$. If $s_n$ is defined (check \eqref{sm1}) by $H^{\le n}=H^{(m_n-1)}+((D^{(m_n)})^{\le s_n})_{\to 3L_{m_n}}$, then \eqref{mul33} and \eqref{cut2} yield
\[
|\frak{r}(z,H^{\le n})-\frak{r}(z,H^{(m_n-1)})|\lesssim \exp(-m_n)\to 0, \quad n\to\infty
\]
uniformly in $z\in \T$. Hence, $\frak{r}(z,H^{\le n})$ converges uniformly on $\T$. 
\end{proof}
We can now finish the proof of Theorem \ref{t1}. First, we claim that $\{\frak{a}(z,H^{\langle n\rangle})\}$ diverges at every point $z\in \T$.
We argue by contradiction. Suppose $\lim_{n\to\infty}\frak{a}(z^*,H^{\langle n\rangle})=\alpha^*$ at some point $z^*\in \T$. Clearly, $|\alpha^*|\ge 1$. Then, for every $\epsilon>0$ there is $n_\epsilon$ such that 
$
|\frak{a}^{(*)}(z^*,H^{\langle n\rangle})- \overline{\alpha^{*}}|\le \epsilon $
for all $n\ge n_\epsilon$. Take $\epsilon=0.1$. By Lemma \ref{argu} in Appendix, we also have
$
|\arg \frak{a}^{(*)}(z^*,H^{\langle n_2\rangle})- \arg \frak{a}^{(*)}(z^*,H^{\langle n_1\rangle})|\lesssim \epsilon
$
for all $n_1,n_2\ge n_\epsilon$. That contradicts Lemma \ref{gr9} and so $\{\frak{a}(z,H^{\langle n\rangle})\}$ diverges at every $z\in \T$. 

Now let $F=\mu \chi_{n=0}+H$, where $\mu\in (0,1)$ is a parameter. Notice that these two terms have disjoint supports since $\supp H\subset \N$. Then, \eqref{rec5} yields
\begin{eqnarray*}
\frak{a}^{(*)}(z,F^{\langle n\rangle})=\frac{\mu}{(1-\mu^2)^{\tfrac 12}}\frak{b}(z,H^{\langle n\rangle})+\frac{1}{(1-\mu^2)^{\tfrac 12}}\frak{a}^{(*)}(z,H^{\langle n\rangle})=\\\frak{a}^{(*)}(z,H^{\langle n\rangle})\left(\frac{1}{(1-\mu^2)^{\tfrac 12}}+\frak{r}(z,H^{\langle n\rangle})\frac{\mu}{(1-\mu^2)^{\tfrac 12}}\right)\,.
\end{eqnarray*}
By the previous lemma, the second factor converges uniformly on $\T$ and the limiting function has no roots on $\T$ as long as $n^*$ is chosen large enough for given $\mu$ because $|\frak{r}(z,H^{\langle n\rangle})|\lesssim \exp(-n^*)$. So $\{\frak{a}^{(*)}(z,F^{\langle n\rangle})\}$ diverges on $\T$. Similarly,
\begin{eqnarray*}
\frak{b}(z,F^{\langle n\rangle})=\frac{1}{(1-\mu^2)^{\tfrac 12}}\frak{b}(z,H^{\langle n\rangle})+\frac{\mu}{(1-\mu^2)^{\tfrac 12}}\frak{a}^{(*)}(z,H^{\langle n\rangle})=\\\frak{a}^{(*)}(z,H^{\langle n\rangle})\left(\frac{\mu}{(1-\mu^2)^{\tfrac 12}}+\frak{r}(z,H^{\langle n\rangle})\frac{1}{(1-\mu^2)^{\tfrac 12}}\right)
\end{eqnarray*}
and the last expression has no limit at every point of $\T$ for the same reason. The Theorem \ref{t1} is proved.
\qed
\medskip

\noindent {\bf Remark.} We emphasize again  that the sequence $H$ we constructed satisfies the following properties:
\begin{eqnarray}\label{sc1}
\supp H\subset \N,\\\label{sc2}
\|H\|_{\ell^2(\Z)}\lesssim \exp(-n^*), \,\, \text{where}\,\, n^*\,\,\text{is arbitrary large},\\
\frak{r}(z,H^{\le n})\to \frak{r}(z,H),\,\label{sc3}
\end{eqnarray}
where the last convergence is uniform over $\T$ and the limiting function we denoted by $\frak{r}(z,H)$ is continuous and satisfies
\begin{equation}\label{stek}
|\frak{r}(z,H)|\lesssim \exp(-n^*)
\end{equation}
for all $z\in \T$.
\medskip

\noindent {\it Proof of Theorem \ref{t4}.}
Our proof for $\Z$  carries over to the case of NLFT on $\R$ word-for-word (the reader can find more details in \cite{arx}). We will only make a few comments. The Theorem \ref{t11} can be replaced by an inverse scattering result where the bijection is established between $q\in \Sch(\R)$ and $\frak{b}(k,q)\in \Sch(\R)$. On $\T$, to create the infinite growth of the argument, we used two proper subarcs: $\T_{\ell}$ and $\T_r$. On $\R$, after we build an analog of $(\nu,\delta,\epsilon)$-daisy for which the maximal function of the argument is large on the interval $I$ of size one, we define the required $q$ by using the recursion in which the employed intervals $\{I_j\}|_{j=1}^\infty$ satisfy two properties:\medskip

$\bullet$ $|I_j|=1, \forall j\in \N$,

$\bullet$ for every $k\in \R$, we have inclusion $k\in I_j$ for infinitely many $j\in \N$. 
\qed\bigskip

\noindent {\bf Remark.} The construction we used in the proof  of Theorem \ref{t1} is flexible enough to produce $F$ that satisfies $F_n\in \R, \forall n\in \Z$, and  Theorem \ref{t2} can be adapted to the case of polynomials orthogonal on the real line. Similarly, one can modify the proof of Theorem \ref{t4} to show that the generalized eigenfunctions $\psi(x,k)$ of Schr\"odinger operator $-\psi''+V(x)\psi=k^2\psi$ with oscillating and square summable potential $V$ do not have to have the standard Jost asymptotics when $x\to\infty$ for a.e. $k$.\bigskip

\section{Connection to OPUC and proofs of Theorem \ref{t2} and \ref{t39}.} 

The polynomials $\{\phi_n(z,\sigma)\}$ satisfy recurrence
(see \cite{bs}, formulas (1.5.23) and (1.5.24))
\[
\left(
\begin{array}{c}
\phi_{n}(z,\sigma)\\
\phi_n^*(z,\sigma)
\end{array}
\right)= \frac{1}{(1-|\gamma_{n-1}|^2)^{\frac 12}}
\left(\begin{array}{cc}
z& -\overline{\gamma_{n-1}}\\
-\gamma_{n-1} z &1
\end{array}\right)\left(
\begin{array}{c}
\phi_{n-1}(z,\sigma)\\
\phi_{n-1}^*(z,\sigma)
\end{array}
\right), \quad \left(
\begin{array}{c}
\phi_{0}(z,\sigma)\\
\phi_0^*(z,\sigma)
\end{array}
\right)=\left(
\begin{array}{c}
1\\
1
\end{array}
\right)\,,
\]
where $n\in \N$ and the recurrence coefficients $\{\gamma_j\}|_{j=0}^\infty$ satisfy $\gamma_j\in \D$. It is known that there is a bijection between probability measures $\sigma$ (with infinite support on $\T$) and such $\{\gamma_j\}\in \D^\infty$. Denote the measure, given by $\{-\gamma_n\}$, by $\widetilde\sigma$. Then, 
\[
\left(
\begin{array}{c}
\phi_{n}(z,\widetilde\sigma)\\
\phi_n^*(z,\widetilde\sigma)
\end{array}
\right)= \frac{1}{(1-|\gamma_{n-1}|^2)^{\frac 12}}
\left(\begin{array}{cc}
z& \overline{\gamma_{n-1}}\\
\gamma_{n-1} z &1
\end{array}\right)\left(
\begin{array}{c}
\phi_{n-1}(z,\widetilde\sigma)\\
\phi_{n-1}^*(z,\widetilde\sigma)
\end{array}
\right), \quad \left(
\begin{array}{c}
\phi_{0}(z,\widetilde\sigma)\\
\phi_0^*(z,\widetilde\sigma)
\end{array}
\right)=\left(
\begin{array}{c}
1\\
1
\end{array}
\right)\,.
\]
Consider the NLFT on $\Z$ with 
\begin{equation}
\label{sirg}
F: F_n=0, n\le 0;\quad F_n=-\overline{\gamma_{n-1}}, n\ge 1.
\end{equation} From \eqref{nach}, we get
\begin{eqnarray*}
\left(\begin{matrix}
\phi_n(z,\sigma)&\phi_n(z,\widetilde \sigma)\\
\phi_n^*(z,\sigma)&-\phi^*_n(z,\widetilde\sigma)
\end{matrix}\right)=
X_n(z,F)\left(\begin{matrix}
1&1\\
1&-1
\end{matrix}\right)=\hspace{4cm}
\\\left(\begin{matrix}
z^n(\frak{a}(z,F^{\le n})+\frak{b}^{(*)}(z,F^{\le n}))&z^n(\frak{a}(z,F^{\le n})-\frak{b}^{(*)}(z,F^{\le n}))\\
\frak{b}(z,F^{\le n})+\frak{a}^{(*)}(z,F^{\le n})&\frak{b}(z,F^{\le n})-\frak{a}^{(*)}(z,F^{\le n})
\end{matrix}\right)\,.
\end{eqnarray*}
Therefore, 
\begin{equation}\label{conn}
\phi_n^*(z,\sigma)=\frak{b}(z,F^{\le n})+\frak{a}^{(*)}(z,F^{\le n}), \quad \phi^*_n(z,\widetilde\sigma)=-\frak{b}(z,F^{\le n})+\frak{a}^{(*)}(z,F^{\le n})\,.
\end{equation}
We will prove a result stronger than Theorem \ref{t2}. It will imply Theorem \ref{t39}, as well.
\begin{theorem}\label{tbig} For every $\epsilon>0$, there is a function $w\in C(\T)$ such that $\|w\|_{1,m}=1$ and $\|w-1\|_{C(\T)}\le \epsilon$ so that the sequence
$
\{\phi_n^*(z,\sigma)\}
$ diverges at every point $z\in \T$. Here, $d\sigma=wdm$.

Moreover, there is a sequence $\{\alpha_n\}\in \ell^2(\Z^+)$ such that the othogonal series 
\[
\sum_{n\ge 0}\alpha_n\phi_n(z,\sigma)
\]
diverges at every $z\in \T$.
\end{theorem}
\begin{proof}Take $F=H$ from \eqref{sc1}-\eqref{sc3}. For such $F$, define $\{\gamma_n\}$ and then $\sigma$ as above in \eqref{sirg}. From \eqref{conn}, one has
\[
\phi_n^*(z,\sigma)=\frak{a}^{(*)}(z,F^{\le n})(1+\frak{r}(z,F^{\le n}))\,.
\]
By Theorem \ref{t1} and Lemma \ref{sss}, the sequence $\{\phi_n^*(z,\sigma)\}$ diverges at every $z\in \T$. Now, we only need to show that $\sigma$ satisfies the required properties when $n^*$ is chosen large enough. Define the Wall polynomials, related to $\sigma$, by $A_n(z,\sigma)$ and $B_n(z,\sigma)$ (see \cite{khr}). Denote by $f(z,\sigma)$ the Schur function for $\sigma$. Then, we have (see \eqref{conn} and the formula (5.5) in \cite{khr}).
\begin{equation}\label{ljg}
\frak{r}(z,F^{\le n+1})=\frac{\frak{b}(z,F^{\le n+1})}{\frak{a}^{(*)}(z,F^{\le n+1})}=-z\frac{A_n(z,\sigma)}{B_n(z,\sigma)}\,.
\end{equation}
By $\eqref{sc3}$, we have $\frak{r}(z,F^{\le n+1})\to \frak{r}(z,F)$ uniformly over $\T$. The right hand side in \eqref{ljg} is analytic in $\D$ so we can extend this uniform convergence to $\D$. Since ${A_n(z,\sigma)}/{B_n(z,\sigma)}\to f(z,\sigma)$ locally uniformly in $\D$ (check Corollary 4.7 in \cite{khr}), we get $ \frak{r}(z,F)=-zf(z,\sigma)$, $f(z,\sigma)\in A(\D)$ and $\|f(z,\sigma)\|_{A(\D)}\lesssim e^{-n^*}$ by \eqref{stek}. By the formula (2.2) in \cite{khr}, we get
\[
w:=\sigma'(z)=\frac{1-|f(z,\sigma)|^2}{|1-zf(z,\sigma)|^2},\quad z\in \T
\]
and $|w-1|\lesssim e^{-n^*}$. Making $n^*$ large enough, we ensure the required bound $\|w-1\|_{C(\T)}\le \epsilon$. We now focus on the second claim of the theorem.
From the OPUC recurrence, we get
\[
\nu_n\phi^*_n(z,\sigma)=1-z\sum_{j=0}^{n-1}\gamma_j\nu_j\phi_j(z,\sigma), \quad \nu_n=\prod_{j\le n-1}(1-|\gamma_j|^2)^{\frac 12}\,,
\]
which is a partial sum of the orthogonal series $\{\phi_j\}|_{j=0}^\infty$ in $L^2_\sigma(\T)$. For Szeg\H{o} measures, $\lim_{n\to\infty}\nu_n$ converges to a positive number. We take $\alpha_j=\gamma_j\nu_j$. Since $\{\alpha_j\}\in \ell^2(\Z^+)$ and $\{\phi_n^*(z,\sigma)\}$ diverges, the claim follows.
\end{proof}\smallskip

\noindent {\bf Other results on pointwise convergence/divergence.}\medskip

Below, we list some applications  of  general results by Menshov and Olicz on orthogonal series  (see  \cite{suetin}).  In what follows, $\{\gamma_n\}$ are the recursion parameters of OPUC with measure $\sigma$. \medskip

$\bullet$ (see \cite{kacz}, p.190) {\it If $\sum_{n\ge 1}|\gamma_n|^2\log^2n<\infty$, then $\lim_{n\to\infty} \phi_n^*(z,\sigma)$ exists a.e. on $\T$.}\medskip

$\bullet$ (see \cite{kacz}, p.201) {\it 
Suppose $\omega(t):\R^+\mapsto \R^+$, $\omega(t) \uparrow +\infty$ as $t\to\infty$, and $\sum_{n\ge 0}\omega^{-1}(n)<\infty$. If $\gamma_n\to 0$ and 
\[
\sum_{n\ge 0}  |\gamma_n|^2\cdot (\log^2|\gamma_n|)\cdot  \omega(\log |\log |\gamma_n||)<\infty,
\]
then $\lim_{n\to\infty} \phi_n^*(z,\sigma)$ exists a.e. on $\T$}. \smallskip

In particular, (see \cite{kacz}, p.200)  $\{\gamma_n\}\in \ell^p(\Z^+), p\in [1,2)$ implies a.e. convergence of $\{\phi_n^*(z,\sigma)\}$. \bigskip

The following result shows the a.e. convergence of $\{\phi_n^*\}$  for lacunary $\ell^2$ recursion parameters. \medskip

$\bullet$ (see \cite{Jelena} and \cite{Den26sparse}) {\it 
Suppose $\{n_j\}|_{j=0}^\infty$ is a subsequence in $\N$ that satisfies 
\begin{equation}
\frac{n_{j+1}}{n_j}\ge q, \quad \forall j\in \Z^+, \quad n_0=1\,,
\end{equation}
where $q\ge 2$. Take $\{\delta_j\}\in \ell^2(\Z^+)$ such that $\|\left\{\delta_j\right\}\|_{\ell^\infty(\Z^+)}<1$. Let $\gamma_{n_j}=\delta_j, \, \forall j\in \Z^+$ and set $\gamma_n=0$ for all other $n\in \Z^+$. Then, $\lim_{n\to\infty} \phi_n^*(z,\sigma)$ exists for a.e. on $\T$}. \smallskip

All these results can be directly applied to show pointwise convergence of $SU(1,1)$ NLFT on $\Z$ after we split the problem to $\{n\ge 0\}$ and $\{n<0\}$. \smallskip

For measures $\sigma$ that are more regular, we list two results (one can find more in \cite{freud,golp}):\smallskip


$\bullet$ (see \cite{seg}, Theorem 12.1.3) {\it If $d\sigma=wdm$ and the modulus of continuity of $w$ satisfies $\omega(h,w)\lesssim (|\log h|)^{-1-\epsilon}$ for some $\epsilon>0$, then $\{\phi_n^*(z,\sigma)\}$ converges uniformly in $z\in \T$.} \medskip

$\bullet$ (see \cite{murman}, p.32) {\it There is a positive continuous $w$ such that $\limsup_{n\to\infty}|\phi_n(z^*,\sigma)|=+\infty$ at some $z^*\in \T$ where $d\sigma=wdm$.}\smallskip

The mentioned result of M. Ambroladze is one of many that deal with the Steklov problem in the theory of orthogonal polynomials. That problem, generally speaking, asks to estimate the size of the orthogonal polynomial given certain assumptions on the orthogonality measure $\sigma$ (see \cite{suetin} and \cite{Rakh1,Den1,Den2,Den3,Den4,Den5,Den6} for more recent results in this direction).  For example, for $\sigma\in \szc$, the question whether $\limsup_{n\to\infty}|\phi_n(z,\sigma)|<\infty$ for a.e. $z\in \T$ is still open (see \cite{avila} where this problem is viewed as a version of Schr\"odinger conjecture in the context of Jacobi matrices). However, we have the following result:

$\bullet$  (see \cite{suetin}, p.36) {\it  If $\sigma\in \szc$, then
\[
|\phi_n(z,\sigma)|=o(1)\log n
\]
for a.e. $z\in \T$.}

\medskip

\section{Appendix.} 
\begin{lemma}\label{lemap1}
Suppose $\{x_n\}, \{y_n\}$ and $\{z_n\}$ are given sequences of nonnegative numbers and
\[
x_{n+1}\le x_n(1+y_n)+z_n,\,\forall n\ge 0\,,
\]
then
\begin{equation}\label{ap1}
x_n\le M_n\left(x_0+\sum_{j=0}^{n-1}z_jM_{j+1}^{-1}\right)\,,
\end{equation}
where $M_n:=\prod_{j=0}^{n-1} (1+y_j), M_{0}:=1$. Moreover, if $x_0\le 1$, $
\|\{y_n\}\|_{\ell^1(\Z^+)}\le \tfrac 12$ and $
\|\{z_n\}\|_{\ell^1(\Z^+)}\le \tfrac 12$, we get
\begin{equation}\label{ap2}
x_n\le x_0+C\|\{y_n\}\|_{\ell^1(\Z^+)}+C\|\{z_n\}\|_{\ell^1(\Z^+)}\,.
\end{equation}
\end{lemma}
\begin{proof}
Let $x_n=M_n\widetilde x_n, n\in \Z^+$. Then, we have
$
\widetilde x_{n+1}\le \widetilde x_n+M_{n+1}^{-1}z_n, \, \widetilde x_0=x_0\,.
$
After adding those bounds, we get
\[
\widetilde x_n\le x_0+\sum_{j=0}^{n-1} M_{j+1}^{-1}z_j\,.
\]
Multiplying  both sides with $M_n$, we get \eqref{ap1}. The estimate  \eqref{ap2} follows.
\end{proof}

\begin{lemma}\label{argu} Suppose the sequence $F$ satisfies two properties: $\supp F_n\subset [n_0,\infty)$ and $F_n\to 0$ as $n\to +\infty$. Then, 
\[
|\arg \frak{a}^{(*)}(z,F^{\langle n\rangle})-\arg \frak{a}^{(*)}(z,F^{\langle n-1\rangle})|\lesssim |F_n|\to 0, \quad n\to\infty
\]
uniformly in $z\in \overline{\D}$.
\end{lemma}
\begin{proof}For $n\ge |n_0|$, we have $F^{\le n}=F^{\langle n\rangle}$. Notice that \eqref{rec5} gives
\begin{eqnarray}
\frak{a}^{(*)}(z,F^{\langle n\rangle})=\frac{1}{(1-|F_n|^2)^{\tfrac 12}}(F_nz^n \frak{b}^{(*)}(z,F^{\langle n-1\rangle})+\frak a^{(*)}(z,F^{\langle n-1\rangle}))
\end{eqnarray}
for $n\ge |n_0|+1$. Then, 
\begin{eqnarray}
\frac{\frak{a}^{(*)}(z,F^{\langle n\rangle})}{\frak a^{(*)}(z,F^{\langle n-1\rangle})}
=\frac{1}{(1-|F_n|^2)^{\tfrac 12}}
\left(\frac{F_nz^n \frak{b}^{(*)}(z,F^{\langle n-1\rangle})}{\frak a^{(*)}(z,F^{\langle n-1\rangle})}+1\right)\,.
\end{eqnarray}
Since
$
\left|{\frak{b}^{(*)}(z,F^{\langle n-1\rangle})}/{\frak a^{(*)}(z,F^{\langle n-1\rangle})}\right|\le 1, \, z\in \T
$
and $F_n\to 0$, we get
\[
\left|\frac{\frak{a}^{(*)}(z,F^{\langle n\rangle})}{\frak a^{(*)}(z,F^{\langle n-1\rangle})}-1\right|
\lesssim |F_n|
\]
uniformly in $z\in \T$ if $n$ is large enough. The maximum principle extends this bound to $z\in \D$. The estimate on the difference of the arguments follows.
\end{proof}
\begin{lemma}\label{var} Suppose the sequence $G$ is such that $\supp G\subset [n_0,\infty)$  and  $\sum_{m+1\le s\le m+p}|G_s|\le \tfrac 12$ for a fixed $m$ and $p\in \N$. Then,
\begin{equation}\label{efim}
|\frak{r}(z,G^{\le m+p})-\frak{r}(z,G^{\le m})|\lesssim \sum_{m+1\le s\le m+p}|G_s|
\end{equation}
for $z\in \T$.
\begin{proof}We take $z\in \T$. For shorthand, we let $\alpha_j(z):=\frak{a}^{(*)}(z,G^{\le j})$, $\beta_j(z):=\frak{b}(z,G^{\le j})$, and $r_j:=\beta_j/\alpha_j$. Then, from \eqref{rec5}, one has
\[
\alpha_{n+1}=\frac{1}{(1-|G_{n+1}|^2)^{\tfrac 12}}(\alpha_n+G_{n+1}z^{n+1}\beta_{n}^*), \, \beta_{n+1}=\frac{1}{(1-|G_{n+1}|^2)^{\tfrac 12}}(\beta_n+G_{n+1}z^{n+1}\alpha_n^*)\,.
\]
That gives
\begin{eqnarray*}
r_{n+1}-r_n=\frac{r_n+G_{n+1}z^{n+1}\alpha_n^*/\alpha_n}{1+G_{n+1}z^{n+1}\beta_n^*/\alpha_n}-r_n=G_{n+1}z^{n+1}\frac{\alpha_n^*/\alpha_n-\beta_n^*r_n/\alpha_n}{1+G_{n+1}z^{n+1}\beta_n^*/\alpha_n}\stackrel{
\eqref{sul}}{=}\\\frac{G_{n+1}z^{n+1}}{\alpha_n^2(1+G_{n+1}z^{n+1}\beta_n^*/\alpha_n)}
\end{eqnarray*}
and, recalling that $|\beta_n^*/\alpha_n|\le 1, |r_n|< 1, |\alpha_n|\ge 1$, we get
\[
=|r_{n+1}-r_n|\le {|G_{n+1}|}/({1-|G_{n+1}|})\,,
\]
so \eqref{efim} follows by summation and triangle inequality.
\end{proof}
\end{lemma}
\bigskip

\bibliographystyle{plain} 
\bibliography{bibfile}
\end{document}